\documentclass[english]{article}
\usepackage[T1]{fontenc}
\usepackage[latin9]{inputenc}
\usepackage{geometry}
\usepackage{color}
\usepackage{babel}
\usepackage{mathrsfs}
\usepackage{url}
\usepackage{amsthm}
\usepackage{amsmath}
\usepackage{amssymb}
\usepackage{esint}
\usepackage[unicode=true,
 bookmarks=true,bookmarksnumbered=false,bookmarksopen=false,
 breaklinks=false,pdfborder={0 0 0},backref=page,colorlinks=true]
 {hyperref}
\hypersetup{pdftitle={Spectral Risk Measures},
 pdfauthor={Alois Pichler}}

\makeatletter

\newcommand{\noun}[1]{\textsc{#1}}

\usepackage{enumitem}		% customizable list environments
\theoremstyle{plain}
\newtheorem{thm}{\protect\theoremname}
  \theoremstyle{definition}
  \newtheorem{defn}[thm]{\protect\definitionname}
  \theoremstyle{remark}
  \newtheorem{rem}[thm]{\protect\remarkname}
  \theoremstyle{plain}
  \newtheorem{cor}[thm]{\protect\corollaryname}
  \theoremstyle{plain}
  \newtheorem{lem}[thm]{\protect\lemmaname}

\usepackage{dsfont}   %	für \one und Erwartungswert   doublestroke
\usepackage{ifpdf}
\ifpdf		
          \usepackage{pdfsync} 	% enable reverse search (inverse search) for PDF 
          \IfFileExists{lmodern.sty}{\usepackage{lmodern}}{}	% bessere PDF-Lesbarkeit
\else		% reverse search (inverse search) for DVI
          \usepackage[active]{srcltx}
\fi

\newcommand{\one}{\mathds 1}  % \usepackage[sans]{dsfont}
\newcommand{\AVaR}{\operatorname{A\kern -0.15em V{\kern -0.15em @}R}}
\newcommand{\VaR}{\operatorname{V{\kern -0.15em @}R}}

\newcommand{\esssup}{\operatornamewithlimits{ess\,sup}}

\makeatother

  \providecommand{\corollaryname}{Corollary}
  \providecommand{\definitionname}{Definition}
  \providecommand{\lemmaname}{Lemma}
  \providecommand{\remarkname}{Remark}
\providecommand{\theoremname}{Theorem}

\begin{document}

\title{Spectral Risk Measures, \\
With Adaptions For Stochastic Optimization }

\author{Alois Pichler~%
\thanks{University of Vienna. Department of Statistics and Operations Research.\protect \\
Contact: \protect\url{alois.pichler@univie.ac.at}%
}}
\maketitle
\begin{abstract}
Stochastic optimization problems often involve the expectation in
its objective. When risk is incorporated in the problem description
as well, then risk measures have to be involved in addition to quantify
the acceptable risk, often in the objective. For this purpose it is
important to have an adjusted, adapted and efficient evaluation scheme
for the risk measure available. In this article different representations
of an important class of risk measures, the spectral risk measures,
are elaborated. The results allow concise problem formulations, they
are particularly adapted for stochastic optimization problems. Efficient
evaluation algorithms can be built on these new results, which finally
make optimization problems involving spectral risk measures eligible
for stochastic optimization.

\end{abstract}

\section{Introduction}

Spectral risk measures have been introduced and studied first in \cite{Acerbi2002}
and \cite{Acerbi2002a}, they represent a risk measure in the general
axiomatic environment introduced by Artzner et al. in the seminal
paper \cite{Artzner1999}. An important study of spectral risk measures,
although under the name\emph{ distortion functional}, was provided
in \cite{Pflug2006}. The concept of spectral risk measures and distortion
functionals is essentially the same -- they differ just in an opposite
sign, resulting in a concave rather than a convex description. But
this additional name \emph{distortion functional} exhibits the interpretation
that some outstanding observations are overvalued, whereas others
are undervalued, creating a somehow distorted overall picture for
the entire observations. 

Spectral risk measures constitute an elementary and important class
of risk measures, as every risk measure can be described as a supremum
of spectral risk measures. They are moreover defined in an explicit
way, hence there is an explicit evaluation scheme available. 

The most important spectral risk measure, which made its way to the
top, is the Average Value-at-Risk. It is the essential advantage of
the Average Value-at-Risk that different representations are known,
which makes this risk measure eligible in varying situations: by conjugate
duality there is an expression in the form of a supremum, but in applications,
particularly in optimization, an equivalent expression as an infimum
is extremely convenient: developed in the paper \emph{Optimization
of Conditional Value-at-Risk} \cite{RockafellarUryasev2000} in this
journal, and brought to beauty in \emph{Some Remarks on the Value-at-Risk
and the Conditional Value-at-Risk} \cite{Pflug2000}. In this article
we extend both formulations to spectral risk measures. A description
as a supremum is a first result of this article. Besides that~--
and this is of importance for applications and the main result~--
a description involving an infimum is elaborated as well. In comparison
with the Average Value-at-Risk there is a price to pay to allow a
representation as an infimum, as the decision space substantially
increases. 

The alternative representation of spectral risk measures is of particular
importance for stochastic programming, as the new formulation allows
a direct approach to solve stochastic optimization problems by employing
spectral risk measures in the objective. Applications in portfolio
optimization (asset allocation) are immediate.

The article is organized as follows. The spectral risk measure is
introduced in Section~\ref{sec:Spectral-Risk-Measures}. Its description
as a supremum is contained in Section~\ref{sec:Dual}, and as an
infimum in Section~\ref{sec:Infimum-Representation}. The implications
for stochastic optimization (asset allocation, e.g.) are outlined
and explained in Section~\ref{sec:StochasticOptimization}.

\section{\label{sec:Spectral-Risk-Measures}Spectral Risk Measures}

Risk measures evaluate on random variables defined on a probability
space $\left(\Omega,\Sigma,P\right)$. In the following definition
we follow the axiomatic environment which has been proposed and formulated
in \cite{Artzner1997}. Often, risk measures are restricted to non-atomic
probability spaces. To make the results eligible for stochastic optimization
and numerical approximations we emphasize that we do not require the
probability space to be atom-less.
\begin{defn}
\label{Def-RiskMeasure}A positively homogeneous \emph{risk measure}
is a mapping $\mathcal{R}\colon L^{p}\rightarrow\mathbb{R}\cup\left\{ \infty\right\} $
with the following properties:
\begin{enumerate}[label=(M)]
\item \emph{\noun{Monotonicity: \label{enu:Monotonicity} }}$\mathcal{R}\left(Y_{1}\right)\le\mathcal{R}\left(Y_{2}\right)$
whenever $Y_{1}\le Y_{2}$ almost surely;
\end{enumerate}

\begin{enumerate}[label=(C)]
\item \emph{\noun{Convexity: \label{enu:Convexity} }}$\mathcal{R}\left(\left(1-\lambda\right)Y_{0}+\lambda Y_{1}\right)\le\left(1-\lambda\right)\mathcal{R}\left(Y_{0}\right)+\lambda\mathcal{R}\left(Y_{1}\right)$
for $0\le\lambda\le1$;
\end{enumerate}

\begin{enumerate}[label=(T)]
\item \emph{\noun{Translation~equivariance}}%
\footnote{In an economic or monetary environment this is often called \emph{\noun{Cash
invariance}} instead.%
}\emph{\noun{: \label{enu:Equivariance} }}$\mathcal{R}\left(Y+c\right)=\mathcal{R}\left(Y\right)+c$
if $c\in\mathbb{R}$;
\end{enumerate}

\begin{enumerate}[label=(H)]
\item \emph{\noun{Positive~homogeneity: \label{enu:Homogeneity} }}$\mathcal{R}\left(\lambda Y\right)=\lambda\mathcal{R}\left(Y\right)$
whenever $\lambda>0$.
\end{enumerate}
\end{defn}
In the convex and monotone setting exposed in Definition~\ref{Def-RiskMeasure}
the random variable $Y$ can be naturally associated with loss, it
is hence common, for example, in an actuarial context in the insurance
industry. In a banking or investment environment the interpretation
of a reward is more natural, in this context the mapping $\rho\colon Y\mapsto\mathcal{R}\left(-Y\right)$
is often considered~-- and called \emph{coherent risk measure~}--
instead (note, that essentially the monotonicity condition \ref{enu:Monotonicity}
and translation property \ref{enu:Equivariance} reverse for $\rho$).

The term \emph{acceptability functional} is often employed in energy
or decision theory to quantify and classify acceptable strategies.
In this context the concave mapping $\mathcal{A}:Y\mapsto-\mathcal{R}\left(-Y\right)$,
the acceptability functional, is employed instead (here, \ref{enu:Convexity}
modifies to concavity). 

\medskip{}

An important class of positively homogeneous risk measures is provided
by spectral risk measures.
\begin{defn}[Spectral Risk Measure]
For $\sigma\in L^{1}\left(0,1\right)$ a non-decreasing probability
density function (i.e.\ $\sigma\ge0$ and $\int_{0}^{1}\sigma(u)\mathrm{d}u=1$)
a mapping
\begin{align}
\mathcal{R}_{\sigma}:L^{1} & \rightarrow\mathbb{R}\cup\left\{ \infty\right\} \nonumber \\
Y & \mapsto\int_{0}^{1}F_{Y}^{-1}(\alpha)\sigma(\alpha)\mathrm{d}\alpha\label{eq:Rsigma}
\end{align}
is called \emph{spectral risk measure} with \emph{spectrum} (or \emph{spectral
density}) $\sigma$.
\end{defn}
$F_{Y}^{-1}(\alpha)$ is the left-continuous, lower semi-continuous
\emph{quantile} or \emph{Value-at-Risk} at level $\alpha$ (sometimes
also \emph{lower inverse} cdf), 
\[
F_{Y}^{-1}(\alpha):=\VaR_{\alpha}\left(Y\right):=\inf\left\{ y\colon P\left(Y\le y\right)\ge\alpha\right\} .
\]

It is evident that for general $\sigma\in L^{1}$ and $Y\in L^{1}$
the risk measure $\mathcal{R}_{\sigma}$ might evaluate to $+\infty$
or $-\infty$. To overcome we will occasionally require $\sigma\in L^{\infty}(0,1)$
or $Y\in L^{\infty}$ and state this explicitly in the text. 
\begin{rem}
For $Y\ge0$ a.s.\ the spectral risk measures can be expressed by
use of $Y$'s cdf ($F_{Y}\left(y\right):=P\left(Y\le y\right)$ directly,
it has the representation 
\begin{equation}
\mathcal{R}_{\sigma}\left(Y\right)=\int_{0}^{\infty}\tau_{\sigma}\left(F_{Y}\left(q\right)\right)\mathrm{d}q,\label{eq:FY}
\end{equation}
where $\tau_{\sigma}=\int_{\alpha}^{1}\sigma\left(p\right)\mathrm{d}p$.
This formulation can be found in insurance (the loss is usually positive),
but as well when associating $Y$ with cost.
\end{rem}
\medskip{}

The Average Value-at-Risk is the most important spectral risk measure. 
\begin{defn}[Average Value-at-Risk]
The spectral risk measure with spectrum 
\begin{equation}
\sigma_{\alpha}\left(\cdot\right):=\frac{1}{1-\alpha}\one_{\left(\alpha,1\right]}\left(\cdot\right)\label{eq:sigma}
\end{equation}
is
\[
\AVaR_{\alpha}\left(Y\right)=\frac{1}{1-\alpha}\int_{\alpha}^{1}\VaR_{p}\left(Y\right)\mathrm{d}p,
\]
the \emph{Average Value-at-Risk} at level $\alpha$ ($0\le\alpha<1$).
The Average Value-at-Risk at level $\alpha=1$ is 
\[
\AVaR_{1}\left(Y\right):=\lim_{\alpha\nearrow1}\AVaR_{\alpha}\left(Y\right)=\esssup Y.
\]

\end{defn}
Due to the defining equation \eqref{eq:Rsigma} of the spectral risk
measure the same real number is assigned to all random variables $Y$
sharing the same law, irrespective of the underlying probability space.
This gives rise to the notion of version independence: 
\begin{defn}
A risk measure $\mathcal{R}$ is \emph{version independent}%
\footnote{sometimes also \emph{law invariant} or \emph{distribution based}.%
}, if $\mathcal{R}\left(Y_{1}\right)=\mathcal{R}\left(Y_{2}\right)$
whenever $Y_{1}$ and $Y_{2}$ share the same law, that is $P\left(Y_{1}\le y\right)=P\left(Y_{2}\le y\right)$
for all $y\in\mathbb{R}$.
\end{defn}
The following representation underlines the central role of the Average
Value-at-Risk for version independent risk measures. Moreover, it
is the basis and justification for investigating spectral risk measures
in much more detail.
\begin{thm}[Kusuoka's representation]
\label{thm:Kusuoka}Any version independent, positively homogeneous
and lower semi-continuous risk measure $\mathcal{R}$ on $L^{\infty}$
of an atom-less probability space has the representation 
\begin{equation}
\mathcal{R}\left(Y\right)=\sup_{\mu\in\mathscr{M}}\int_{0}^{1}\AVaR_{\alpha}\left(Y\right)\mu\left(\mathrm{d}\alpha\right),\label{eq:Kusuoka}
\end{equation}
where $\mathscr{M}$ is a set of probability measures on $\left[0,1\right]$.\end{thm}
\begin{proof}
Cf. \cite{Kusuoka,PflugRomisch2007,Shapiro2011}.
\end{proof}
In the context of spectral risk measures it is essential to observe
that any spectral risk measure has an immediate representation as
in \eqref{eq:Kusuoka}, the measure $\mu_{\sigma}$ corresponding
to the density $\sigma$ is 
\[
\mu_{\sigma}\left(A\right):=\sigma\left(0\right)\delta_{0}\left(A\right)+\int_{A}1-\alpha\:\mathrm{d}\sigma\left(\alpha\right)\qquad(A\subset[0,1],\text{ measurable}).
\]
$\mu_{\sigma}$ is a positive measure since $\sigma$ is non-decreasing,
and integration by parts reveals that it is a probability measure.
Kusuoka's representation is immediate by Riemann-Stieltjes integration
by parts for the set $\mathscr{M}=\left\{ \mu_{\sigma}\right\} $,
as 
\begin{align*}
\int_{0}^{1}\AVaR_{\alpha}\left(Y\right)\mu_{\sigma}\left(\mathrm{d}\alpha\right) & =\sigma\left(0\right)\AVaR_{0}\left(Y\right)+\int_{0}^{1}\frac{1}{1-\alpha}\int_{\alpha}^{1}\VaR_{p}\left(Y\right)\mathrm{d}p\,\left(1-\alpha\right)\mathrm{d}\sigma\left(\alpha\right)\\
 & =\int_{0}^{1}\VaR_{p}\left(Y\right)\sigma\left(p\right)\mathrm{d}p=\mathcal{R}_{\sigma}\left(Y\right).
\end{align*}
Conversely, the risk measure $\int_{0}^{1}\AVaR_{\alpha}\left(Y\right)\mu\left(\mathrm{d}\alpha\right)$
in Kusuoka's representation \eqref{thm:Kusuoka} often can be expressed
as a spectral risk measure with spectral density $\sigma_{\mu}$,
this is accomplished by the function 
\begin{equation}
\sigma_{\mu}(\alpha):=\int_{0}^{\alpha}\frac{1}{1-p}\mu\left(\mathrm{d}p\right).\label{eq:3-1}
\end{equation}
Provided that $\sigma_{\mu}$ is well defined (notice that possibly
$\mu\left(\left\{ 1\right\} \right)>0$ is a problem to compute $\sigma\left(1\right)$)
it is positive and a density, as $\int_{0}^{1}\sigma(\alpha)\mathrm{d}\alpha=\int_{0}^{1}\frac{1}{1-p}\int_{p}^{1}\mathrm{d}\alpha\mu\left(\mathrm{d}p\right)=1$. 

As a matter of fact every risk measure $\mathcal{R}$ allows a representation
in terms of spectral risk measures as 
\[
\mathcal{R}\left(Y\right)=\sup_{\sigma\in\mathcal{S}}\mathcal{R}_{\sigma}\left(Y\right),
\]
where $\mathcal{S}$ is a set of continuous and (except for the expectation,
for which $\mathcal{S}=\left\{ \one\right\} $) strictly increasing
(thus invertible) density functions. A rigorous discussion is rather
straight forward, although beyond the scope of this article.

\section{\label{sec:Dual}Supremum-Representation of Spectral Risk Measures }

The duality theory based on the Legendre transform is a usual way
to develop alternative representations of convex functions. The resulting
representation provides an expression of the spectral risk measure
as a supremum. 
\begin{defn}[Dual function]
The \emph{dual function} (convex conjugate) of an risk measure $\mathcal{R}$
is defined as 
\[
\mathcal{R}^{*}\left(Z\right)=\sup_{Y\in L^{\infty}}\mathbb{E}YZ-\mathcal{R}\left(Y\right)\qquad\left(Z\in L^{1}\right).
\]
A random variable $Z$ is called \emph{feasible} for $\mathcal{R}$
if 
\[
\mathcal{R}^{*}\left(Z\right)<\infty.
\]

\end{defn}
It follows from the Fenchel-Moreau duality theorem (cf. \cite{RuszczynskiShapiro2009})
that %
\footnote{Here, $L^{\infty}$ is paired with (its pre-dual) $L^{1}$ for convenience.
Different pairings are possible as well, and adapted to particular
problems.%
} 
\[
\mathcal{R}\left(Y\right)=\sup_{Z\in L^{1}}\mathbb{E}YZ-\mathcal{R}^{*}\left(Z\right),
\]
provided that $\mathcal{R}$ is lower semi-continuous (cf. \cite{Ruszczynski2006}
and \cite{Schachermayer2006} for characterizations on lower semi-continuity).
From positive homogeneity it is immediate that $Z$ is feasible iff
$\mathcal{R}^{*}\left(Z\right)=0$. 

\bigskip{}

To formulate the following representation result in a concise way
we employ the notion of convex ordering. We basically follow \cite{Shapiro2011}.
\begin{defn}[Convex ordering]
Let $\tau,\,\sigma:\left[0,1\right]\rightarrow\mathbb{R}$ be integrable
functions.
\begin{enumerate}[label=(\roman{enumi})]
\item $\sigma$ majorizes $\tau$ (denoted $\sigma\succcurlyeq\tau$ or
$\tau\preccurlyeq\sigma$ ) iff 
\[
\int_{\alpha}^{1}\tau\left(p\right)\mathrm{d}p\le\int_{\alpha}^{1}\sigma\left(p\right)\mathrm{d}p\quad\left(\alpha\in\left[0,1\right]\right)\mbox{ and }\int_{0}^{1}\tau\left(p\right)\mathrm{d}p=\int_{0}^{1}\sigma\left(p\right)\mathrm{d}p.
\]

\item The spectrum $\sigma$ majorizes the random variable $Z$ ($Z\preccurlyeq\sigma$)
iff 
\[
(1-\alpha)\AVaR_{\alpha}\left(Z\right)\le\int_{\alpha}^{1}\sigma\left(p\right)\mathrm{d}p\quad\left(\alpha\in\left[0,1\right]\right)\quad\mbox{ and }\mathbb{E}Z=\int_{0}^{1}\sigma\left(p\right)\mathrm{d}p.
\]

\end{enumerate}
\end{defn}
\begin{rem}
It should be noted that 
\[
Z\preccurlyeq\sigma\mbox{ if and only if }p\mapsto\VaR_{p}\left(Z\right)\preccurlyeq\sigma.
\]
Moreover $Z\preccurlyeq\sigma$ is related to a \emph{convex order
constraint} condition (cf. \cite{StoyanMueller2002} or \cite{shanked}).
The dominance in convex (concave) order was used in studying risk
measures for example in \cite{Follmer2004,Dana2005}. 
\end{rem}
The following Theorem~\ref{thm1} is a characterization by employing
the convex conjugate relationship. \pagebreak[1] 
\begin{thm}[Representation as a Supremum -- Dual Representation of Spectral Risk
Measures]
\label{thm1}Let $\mathcal{R}_{\sigma}\left(Y\right)$ be a spectral
risk measure.  Then 
\begin{enumerate}[label=(\roman{enumi})]
\item the representation 
\begin{align}
\mathcal{R}_{\sigma}\left(Y\right) & =\sup\left\{ \mathbb{E}\, YZ\colon Z\preccurlyeq\sigma\right\} \label{eq:Dual}\\
 & =\sup\left\{ \mathbb{E}\, YZ\colon\mathbb{E}Z=1,\,\left(1-\alpha\right)\AVaR_{\alpha}\left(Z\right)\le\int_{\alpha}^{1}\sigma\left(p\right)\mathrm{d}p,\,0\le\alpha<1\right\} \nonumber 
\end{align}
holds true, 
\item and
\begin{equation}
\mathcal{R}_{\sigma}\left(Y\right)=\sup\left\{ \mathbb{E}\, Y\cdot\sigma\left(U\right)\colon\, U\mbox{ is uniformly distributed}\right\} ,\label{eq:RsigmaU}
\end{equation}
where the infimum is attained if $Y$ and $U$ are coupled in a co-monotone
way. %
\footnote{$U$ is \emph{uniformly distributed} if $P\left(U\le u\right)=u$
for all $u\in\left[0,1\right]$.%
}
\end{enumerate}
\end{thm}
\begin{rem}
For the measure $\mu_{\sigma}$ associated with $\sigma$ it holds
that $\int_{\alpha}^{1}\sigma\left(p\right)\mathrm{d}p=\int_{0}^{1}\min\left\{ \frac{1-\alpha}{1-p},1\right\} \mathrm{\mu_{\sigma}\left(\mathrm{d}p\right)}$,
hence \eqref{eq:Dual} can be stated equivalently as 
\begin{align*}
\mathcal{R}_{\sigma}\left(Y\right) & =\sup\left\{ \mathbb{E}\, YZ\colon\mathbb{E}Z=1,\,\AVaR_{\alpha}\left(Z\right)\le\int_{0}^{1}\min\left\{ \frac{1}{1-\alpha},\frac{1}{1-p}\right\} \mu_{\sigma}\left(\mathrm{d}p\right),\,0\le\alpha<1\right\} .
\end{align*}

\end{rem}

\begin{rem}
The second statement of Theorem~\ref{thm1} implicitly and tacitly
assumes that the probability space is rich enough to carry a uniform
random variable. This is certainly the case if the probability space
does not contain atoms. But even if the probability space has atoms,
then this is not a restriction neither, as any probability space with
atoms can be augmented to allow a uniformly distributed random variable. 

The second statement of Theorem~\ref{thm1} is already contained
in \cite{PflugRomisch2007}. The theorem moreover is related to the
paper \cite{DentchevaRusz2004}.\end{rem}
\begin{proof}[Proof of Theorem~\ref{thm1}]
Recall the Legendre-Fenchel transformation for convex functions (cf.
\cite{RuszczynskiShapiro2009}) 
\begin{eqnarray}
\mathcal{R}_{\sigma}\left(Y\right) & = & \sup_{Z\in L^{1}}\mathbb{E}\, YZ-\mathcal{R}_{\sigma}^{*}\left(Z\right),\nonumber \\
\mathcal{R}_{\sigma}^{*}\left(Z\right) & = & \sup_{Y\in L^{\infty}}\mathbb{E}\, YZ-\mathcal{R}_{\sigma}\left(Y\right).\label{eq:Dual-2}
\end{eqnarray}
As $\mathcal{R}_{\sigma}$ is version independent the random variable
$Y$ minimizing \eqref{eq:Dual-2} is coupled in a  co-monotone way
with $Z$ (cf. \cite{Hoeffding} and \cite[Proposition 1.8]{PflugRomisch2007}
for the respective rearrangement inequality, sometimes (cf. \cite{Dana2005})
referred to as \emph{Hardy and Littlewood's inequality} or \emph{Hardy-Littlewood-Pólya
inequality}). It follows that 
\begin{align*}
\mathcal{R}_{\sigma}^{*}\left(Z\right) & =\sup_{Y}\mathbb{E}YZ-\mathcal{R}_{\sigma}\left(Y\right)\\
 & =\sup\int_{0}^{1}F_{Y}^{-1}\left(\alpha\right)F_{Z}^{-1}\left(\alpha\right)\mathrm{d}\alpha-\int_{0}^{1}F_{Y}^{-1}\left(\alpha\right)\sigma\left(\alpha\right)\mathrm{d}\alpha,
\end{align*}
the infimum being among all cumulative distribution functions $F_{Y}\left(y\right)=P\left(Y\le y\right)$
of $Y$. Define $G\left(\alpha\right):=\int_{\alpha}^{1}F_{Z}^{-1}\left(p\right)\mathrm{d}p$
and $S\left(\alpha\right):=\int_{\alpha}^{1}\sigma\left(p\right)\mathrm{d}p$,
whence 
\begin{align}
\mathcal{R}_{\sigma}^{*}\left(Z\right) & =\sup_{F_{Y}}\int_{0}^{1}F_{Y}^{-1}\left(\alpha\right)\mathrm{d}\left(S\left(\alpha\right)-G\left(\alpha\right)\right)\nonumber \\
 & =\sup_{F_{Y}}\left[F_{Y}^{-1}\left(\alpha\right)\left(S\left(\alpha\right)-G\left(\alpha\right)\right)\right]_{\alpha=0}^{1}-\int_{0}^{1}S\left(\alpha\right)-G\left(\alpha\right)\mathrm{d}F_{Y}^{-1}\left(\alpha\right)\nonumber \\
 & =\sup_{F_{Y}}F_{Y}^{-1}\left(0\right)\left(G\left(0\right)-S\left(0\right)\right)+\int_{0}^{1}G\left(\alpha\right)-S\left(\alpha\right)\mathrm{d}F_{Y}^{-1}\left(\alpha\right)\label{eq:4}
\end{align}
by integration by parts of the Riemann-Stieltjes integral and as $-\left\Vert Y\right\Vert _{\infty}\le F_{Y}^{-1}\le\left\Vert Y\right\Vert _{\infty}$. 

Consider the constant random variables $Y\equiv c$ ($c\in\mathbb{R}$),
then $F_{Y}^{-1}\equiv c$ and, by \eqref{eq:4}, 
\[
\mathcal{R}_{\sigma}^{*}\left(Z\right)\ge\sup_{c\in\mathbb{R}}c\left(G\left(0\right)-S\left(0\right)\right).
\]
Note now that $S\left(0\right)=\int_{0}^{1}\sigma\left(p\right)\mathrm{d}p=1$,
whence 
\begin{eqnarray*}
\mathcal{R}_{\sigma}^{*}\left(Z\right) & \ge & \sup_{c\in\mathbb{R}}c\left(G\left(0\right)-1\right)=\begin{cases}
0 & \mbox{if }G\left(0\right)=1\\
\infty & \mbox{else}
\end{cases}\quad=\begin{cases}
0 & \mbox{if }\mathbb{E}Z=1\\
\infty & \mbox{else},
\end{cases}
\end{eqnarray*}
because 
\begin{equation}
G\left(0\right)=\int_{0}^{1}F_{Z}^{-1}\left(p\right)\mathrm{d}p=\mathbb{E}Z.\label{eq:11}
\end{equation}
Assuming  $\mathbb{E}Z=1$ it follows from \eqref{eq:4} that 
\[
\mathcal{R}_{\sigma}^{*}\left(Z\right)=\sup_{F_{Y}}\int_{0}^{1}G\left(\alpha\right)-S\left(\alpha\right)\mathrm{d}F_{Y}^{-1}\left(\alpha\right).
\]

Then choose an arbitrary measurable set $B$ and consider the random
variable $Y_{c}:=c\cdot\one_{B^{\complement}}$ ($c>0$). Note that
$F_{Y_{c}}^{-1}=\one_{\left[\alpha_{0},1\right]}$, where $\alpha_{0}=P\left(B\right)$.
With this choice 
\begin{eqnarray*}
\mathcal{R}_{\sigma}^{*}\left(Z\right) & \ge & \sup_{F_{Y_{c}}}\int_{0}^{1}G\left(\alpha\right)-S\left(\alpha\right)\mathrm{d}F_{Y_{c}}^{-1}\left(\alpha\right)\ge\sup_{c\ge0}c\left(G\left(\alpha_{0}\right)-S\left(\alpha_{0}\right)\right)=\\
 & = & \begin{cases}
0 & \mbox{if }G\left(\alpha{}_{0}\right)\le S\left(\alpha_{0}\right)\\
\infty & \mbox{else}
\end{cases}
\end{eqnarray*}
As $B$ was chosen arbitrarily it follows that $G\left(\alpha\right)\le S\left(\alpha\right)$
has to hold for any $0\le\alpha\le1$ for $Z$ to be feasible.

Conversely, if \eqref{eq:11} and $G\left(\alpha\right)\le S\left(\alpha\right)$
for all $0\le\alpha\le1$, then 
\[
\sup_{F_{Y}}\int_{0}^{1}G\left(\alpha\right)-S\left(\alpha\right)\mathrm{d}F_{Y}^{-1}\left(\alpha\right)\le0,
\]
because $\alpha\mapsto F_{Y}^{-1}\left(\alpha\right)$ is a non-decreasing
function. Note now that 
\begin{align*}
\int_{\alpha}^{1}\sigma\left(p\right)\mathrm{d}p & =S\left(\alpha\right)\ge G\left(\alpha\right)\\
 & =\int_{\alpha}^{1}F_{Z}^{-1}\left(p\right)\mathrm{d}p=\left(1-\alpha\right)\AVaR_{\alpha}\left(Z\right),
\end{align*}
from which finally follows that 
\[
\mathcal{R}_{\sigma}^{*}\left(Z\right)=\begin{cases}
0 & \mbox{if }\mathbb{E}Z=1\mbox{ and }\left(1-\alpha\right)\AVaR_{\alpha}\left(Z\right)\le\int_{\alpha}^{1}\sigma\left(p\right)\mathrm{d}p\ \left(0\le\alpha\le1\right)\\
\infty & \mbox{else}.
\end{cases}
\]

As for the second assertion of the theorem consider $Z=\sigma\left(U\right)$
for a uniformly distributed random variable $U$, then $P\left(Z\le\sigma\left(\alpha\right)\right)=P\left(\sigma\left(U\right)\le\sigma\left(\alpha\right)\right)\ge P\left(U\le\alpha\right)=\alpha$,
that is $\VaR_{\alpha}\left(Z\right)\ge\sigma\left(\alpha\right)$.
But as $1=\int_{0}^{1}\sigma\left(\alpha\right)\mathrm{d}\alpha\le\int_{0}^{1}\VaR_{\alpha}\left(\sigma\left(U\right)\right)\mathrm{d}\alpha=\mathbb{E}\sigma\left(U\right)=1$
it follows that 
\[
\sigma\left(\alpha\right)=\VaR_{\alpha}\left(Z\right)
\]
 almost everywhere. Observe now that any $Z$ with $\VaR_{\alpha}\left(Z\right)\le\sigma\left(\alpha\right)$
is feasible, because 
\[
\int_{\alpha}^{1}\sigma\left(p\right)\mathrm{d}p\ge\int_{\alpha}^{1}\VaR_{p}\left(Z\right)\mathrm{d}p=(1-\alpha)\AVaR_{\alpha}\left(Z\right)
\]
and $\mathbb{E}Z=\mathbb{E}\sigma\left(U\right)=\int_{0}^{1}\sigma\left(\alpha\right)\mathrm{d}\alpha=1$.
Now let $U$ be coupled in an co-monotone way with $Y$, then $\mathbb{E}YZ=\int_{0}^{1}F_{Y}^{-1}\left(\alpha\right)F_{Z}^{-1}\left(\alpha\right)\mathrm{d}\alpha=\int_{0}^{1}F_{Y}^{-1}\left(\alpha\right)\VaR_{\alpha}\left(\sigma\left(U\right)\right)\mathrm{d}\alpha=\int_{0}^{1}F_{Y}^{-1}\left(\alpha\right)\sigma\left(\alpha\right)\mathrm{d}\alpha$
such that
\[
\mathcal{R}_{\sigma}\left(Y\right)=\sup\left\{ \mathbb{E}\, Y\sigma\left(U\right)\colon U\mbox{ uniformly distributed}\right\} ,
\]
which is finally the second assertion.\end{proof}
\begin{rem}
We emphasize that the conditions $(1-\alpha)\AVaR_{\alpha}\left(Z\right)\le\int_{\alpha}^{1}\sigma\left(p\right)\mathrm{d}p$
and $\mathbb{E}Z=1$ together imply that $Z\ge0$ almost everywhere.
Indeed, suppose that $P\left(Z<0\right)=:p>0$. Then $1=\mathbb{E}Z=\int_{\left\{ Z<0\right\} }Z\mathrm{d}P+\int_{\left\{ Z\ge0\right\} }Z\mathrm{d}P=\int_{\left\{ Z<0\right\} }Z\mathrm{d}P+\left(1-p\right)\AVaR_{p}(Z)$.
As $\int_{\left\{ Z<0\right\} }Z\mathrm{d}P<0$ it follows that $\left(1-p\right)\AVaR_{p}(Z)>1$.
But this contradicts the fact that $(1-p)\AVaR_{p}\left(Z\right)\le\int_{p}^{1}\sigma\left(p^{\prime}\right)\mathrm{d}p^{\prime}\le1$,
hence $Z\ge0$ almost surely.
\end{rem}
The characterization derived in the previous theorem naturally applies
to the Average Value-at-Risk. This can be simplified further for the
Average Value-at-Risk, as the following corollary exhibits. Whereas
the first statement is just folklore, the second statement is listed
for completeness and because it allows interesting interpretations. 
\begin{cor}
\label{cor:AVaR}The Average Value-at-Risk at level $\alpha$ obeys
the dual representations 
\begin{align*}
\AVaR_{\alpha}\left(Y\right) & =\sup\left\{ \mathbb{E}YZ:\,\mathbb{E}Z=1,\,0\le Z,\,(1-\alpha)Z\le\one\right\} \\
 & =\sup\left\{ \mathbb{E}YZ:\,\mathbb{E}Z=1,\,\AVaR_{p}\left(Z\right)\le\frac{1}{1-\alpha}\mbox{ for all }p>\alpha\right\} .
\end{align*}
\end{cor}
\begin{proof}
The Average Value-at-Risk at level $\alpha$ is provided by the Dirac
measure $\mu_{\alpha}\left(A\right):=\delta_{\alpha}\left(A\right)=\begin{cases}
1 & \mbox{if }\alpha\in A\\
0 & \text{otherwise}
\end{cases}$, and the respective spectral density function is $\sigma_{\alpha}$
(cf. \eqref{eq:sigma}). It follows from $\int_{p}^{1}\sigma_{\alpha}\left(p^{\prime}\right)\mathrm{d}p^{\prime}=\min\left\{ 1,\,\frac{1-p}{1-\alpha}\right\} $
and Theorem~\ref{thm1} that 

\[
\AVaR_{\alpha}\left(Y\right)=\inf\left\{ \mathbb{E}YZ:\,\mathbb{E}Z=1,\,\AVaR_{p}\left(Z\right)\le\min\left\{ \frac{1}{1-p},\frac{1}{1-\alpha}\right\} \right\} .
\]
Observe next that for $Z\ge0$ 
\begin{align*}
\frac{1}{1-p}=\frac{1}{1-p}\mathbb{E}Z & \ge\frac{1}{1-p}\int_{p}^{1}F_{Z}^{-1}\left(p^{\prime}\right)\mathrm{d}p^{\prime}=\AVaR_{p}\left(Z\right),
\end{align*}
hence
\[
\AVaR_{\alpha}\left(Y\right)=\inf\left\{ \mathbb{E}YZ:\,\mathbb{E}Z=1,\,\AVaR_{p}\left(Z\right)\le\frac{1}{1-\alpha}\right\} .
\]
For $p\le\alpha$, in addition, $\AVaR_{p}\left(Z\right)\le\frac{1}{1-p}\le\frac{1}{1-\alpha}$. 

This proves the second assertion.

As for the first observe that $\frac{1}{1-\alpha}\ge\AVaR_{p}\left(Z\right)\rightarrow\esssup Z$,
hence $\left(1-\alpha\right)Z\le\one$; conversely, if $0\le Z$ and
$(1-\alpha)Z\le\one$, then 
\[
\frac{1}{1-\alpha}\ge\esssup Z\ge\AVaR_{p}\left(Z\right),
\]
which is the first assertion.\end{proof}
\begin{cor}
The conditions $\mathbb{E}Z=1$ and $\AVaR_{\alpha}\left(Z\right)\ge1$
(for all $0\le\alpha\le1$) are equivalent to $Z\equiv\one$.\end{cor}
\begin{proof}
The proof is immediate for the particular choice $\alpha=0$ in Corollary~\ref{cor:AVaR}. 
\end{proof}

\section{\label{sec:Infimum-Representation}Infimum-Representation of Spectral
Risk Measures}

The latter Theorem~\ref{thm1} exposes the spectral risk measure
as a supremum. This expression \eqref{eq:Dual}, nor the initial defining
equation \eqref{eq:Rsigma}, nor \eqref{eq:FY} are of good use for
stochastic optimization (cf. the next section). For this it is desirable
to have an expression as an infimum available: the following theorem,
the main result of this article, provides a helpful alternative. The
representation extends the well-known formula for the Average Value-at-Risk
provided in \cite{RockafellarUryasev2000}, finally stated in the
present form in \cite{Pflug2000}. 
\begin{thm}[Representation as an Infimum]
\label{thm:InfRep}For any $Y\in L^{\infty}$ the spectral risk measure
with spectrum $\sigma$ has the representation 
\begin{equation}
\mathcal{R}_{\sigma}\left(Y\right)=\inf_{f}\,\mathbb{E}\, f\left(Y\right)+\int_{0}^{1}f^{*}\left(\sigma\left(p\right)\right)\mathrm{d}p,\label{eq:HS}
\end{equation}
where the infimum is among all arbitrary, measurable functions $f\colon\mathbb{R}\rightarrow\mathbb{R}$
and $f^{*}$ is $f$'s convex conjugate function%
\footnote{The convex conjugate function of $f$  is $f^{*}\left(y\right):=\sup x\cdot y-f\left(x\right)$.
The convex conjugate may evaluate to $+\infty$.%
}.

\end{thm}
The statement of the Inf-Representation Theorem~\ref{thm:InfRep}
can be formulated equivalently by the following forms.
\begin{cor}
\label{cor:5}For any $Y\in L^{\infty}$ the spectral risk measure
with spectrum $\sigma$ allows the representations 
\begin{eqnarray*}
\mathcal{R}_{\sigma}\left(Y\right) & = & \inf_{f\text{ convex}}\,\mathbb{E}\, f\left(Y\right)+\int_{0}^{1}f^{*}\left(\sigma\left(p\right)\right)\mathrm{d}p\\
 & = & \inf\left\{ \mathbb{E}f\left(Y\right):\int_{0}^{1}f^{*}\left(\sigma\left(p\right)\right)\mathrm{d}p\le0\right\} ,
\end{eqnarray*}
where the infimum is among arbitrary, measurable functions $f\colon\mathbb{R}\rightarrow\mathbb{R}$.\end{cor}
\begin{proof}[Proof of Corollary~\ref{cor:5}]
It is well-known that the bi-conjugate function $f^{**}:=\left(f^{*}\right)^{*}$
is a convex and lower semi-continuous function satisfying $f^{**}\le f$
and $f^{***}=f^{*}$ (cf. the analogous Fenchel-Moreau Theorem and
equation \eqref{eq:Dual-2}). The infimum in \eqref{eq:HS} hence
-- without any loss of generality -- can be restricted to \emph{convex}
functions, that is 
\[
\mathcal{R}_{\sigma}\left(Y\right)=\inf_{f\text{ convex}}\,\mathbb{E}f\left(Y\right)+\int_{0}^{1}f^{*}\left(\sigma\left(p\right)\right)\mathrm{d}p.
\]

As for the second assertion notice first that  clearly 
\begin{eqnarray*}
\mathcal{R}_{\sigma}\left(Y\right) & \le & \inf\left\{ \mathbb{E}f\left(Y\right)+\int_{0}^{1}f^{*}\left(\sigma\left(p\right)\right)\mathrm{d}p:\int_{0}^{1}f^{*}\left(\sigma\left(p\right)\right)\mathrm{d}p\le0\right\} \\
 & \le & \inf\left\{ \mathbb{E}f\left(Y\right):\int_{0}^{1}f^{*}\left(\sigma\left(p\right)\right)\mathrm{d}p\le0\right\} .
\end{eqnarray*}
Consider $f_{\alpha}\left(x\right):=f(x)-\alpha$ (where $\alpha$
a constant and $f$ arbitrary). It holds that $f_{\alpha}^{*}\left(y\right)=f^{*}(y)+\alpha$,
as exposed by the auxiliary Lemma~\ref{lem:Transform} in the Appendix.
Hence $\int_{0}^{1}f_{\alpha}^{*}\left(\sigma\left(p\right)\right)\mathrm{d}p=\int_{0}^{1}f^{*}\left(\sigma\left(p\right)\right)\mathrm{d}p+\alpha$
and 
\begin{equation}
\mathbb{E}\, f_{\alpha}\left(Y\right)+\int_{0}^{1}f_{\alpha}^{*}\left(\sigma\left(p\right)\right)\mathrm{d}p=\mathbb{E}\, f\left(Y\right)+\int_{0}^{1}f^{*}\left(\sigma\left(p\right)\right)\mathrm{d}p.\label{eq:12}
\end{equation}
Choose $\alpha:=\int_{0}^{1}f^{*}\left(\sigma\left(p\right)\right)\mathrm{d}p$
such that $\int_{0}^{1}f_{\alpha}^{*}\left(\sigma\left(p\right)\right)\mathrm{d}p=0$.
$f_{\alpha}$ hence is feasible for \eqref{eq:HS} with the same objective
as $f$ by \eqref{eq:12}, from which the assertion follows.
\end{proof}

\begin{rem}
Having a look at representation \eqref{eq:HS} it is not immediate
that the axioms of Definition~\ref{Def-RiskMeasure} are satisfied
-- except for monotonicity, \ref{enu:Monotonicity}. The transformations
listed in Lemma~\ref{lem:Transform} in the Appendix can be used
in a straight forward manner to deduce the properties directly from
\eqref{eq:HS}. 
\end{rem}

\begin{rem}
Notice that $\sigma$ has its range in the interval $\left\{ \sigma\left(x\right):\, x\in\left[0,1\right]\right\} =\left[0,\sigma\left(1\right)\right]$,
and from convexity of $f^{*}$ it follows that the set $\left\{ f^{*}<\infty\right\} $
is convex. Hence $f^{*}\left(y\right)<\infty$ necessarily has to
hold for all $y\in\left(0,\sigma\left(1\right)\right)$ to ensure
that $\int_{0}^{1}f^{*}\left(\sigma(u)\right)\mathrm{d}u<\infty$.
For $f$ convex this means in turn that 
\[
\lim_{x\to-\infty}f^{\prime}\left(x\right)\le0\text{ and }\lim_{x\to\infty}f^{\prime}\left(x\right)\ge\sigma\left(1\right),
\]
limiting thus the class of interesting functions in Corollary~\ref{cor:5}
to convex functions satisfying $f^{\prime}\left(\mathbb{R}\right)\supset\left(0,\sigma(1)\right)$. \end{rem}
\begin{proof}[Proof of Theorem \ref{thm:InfRep}]
 From the definition of the convex conjugate $f^{*}$ it is immediate
that 
\[
f^{*}\left(\sigma\right)\ge y\cdot\sigma-f\left(y\right)
\]
for all numbers $y$ and $\sigma$, hence 
\[
f\left(Y\right)+f^{*}\left(\sigma\left(U\right)\right)\ge Y\cdot\sigma\left(U\right),
\]
where $U$ is any uniformly distributed random variable, i.e.\ $U$
satisfies $P\left(U\le u\right)=u$. Taking expectations it follows
that 
\[
\mathbb{E}f\left(Y\right)+\mathbb{E}f^{*}\left(\sigma\left(U\right)\right)\ge\mathbb{E}\, Y\cdot\sigma\left(U\right).
\]
As $U$ is uniformly distributed it holds that 
\[
\mathbb{E}f^{*}\left(\sigma\left(U\right)\right)=\int_{0}^{1}f^{*}\left(\sigma\left(u\right)\right)\mathrm{d}u,
\]
such that 
\[
\mathbb{E}f\left(Y\right)+\int_{0}^{1}f^{*}\left(\sigma\left(u\right)\right)\mathrm{d}u\ge\mathbb{E}\, Y\cdot\sigma\left(U\right),
\]
irrespective of the uniform random variable $U$. Hence, by $\eqref{eq:RsigmaU}$
in Theorem~\ref{thm1}, 
\[
\mathbb{E}f\left(Y\right)+\int_{0}^{1}f^{*}\left(\sigma\left(u\right)\right)\mathrm{d}u\ge\sup_{U\text{ uniform}}\mathbb{E}\, Y\cdot\sigma\left(U\right)=\mathcal{R}_{\sigma}\left(Y\right),
\]
 establishing the inequality 
\[
\mathcal{R}_{\sigma}\left(Y\right)\le\mathbb{E}f\left(Y\right)+\int_{0}^{1}f^{*}\left(\sigma\left(u\right)\right)\mathrm{d}u.
\]

As for the converse inequality consider the function 
\[
f_{0}\left(y\right):=\int_{0}^{1}F_{Y}^{-1}\left(\alpha\right)+\frac{1}{1-\alpha}\left(y-F_{Y}^{-1}\left(\alpha\right)\right)_{+}\mu_{\sigma}\left(\mathrm{d}\alpha\right).
\]
$f_{0}\left(y\right)$ is well-defined for all $y$ because $Y\in L^{\infty}$;
$f_{0}\left(y\right)$ is moreover increasing and convex, because
$y\mapsto\left(y-q\right)_{+}$ is increasing and convex, and because
$\mu_{\sigma}$ is positive. 

Recall the formula 
\[
\AVaR_{\alpha}\left(Y\right)=\inf_{q\in\mathbb{R}}\: q+\frac{1}{1-\alpha}\mathbb{E}\left(Y-q\right)_{+}
\]
and the fact that the infimum is attained at $q=F_{Y}^{-1}\left(\alpha\right)$
(cf. \cite{Pflug2000} for the general  formula), providing thus the
explicit form 
\[
\AVaR_{\alpha}\left(Y\right)=F_{Y}^{-1}\left(\alpha\right)+\frac{1}{1-\alpha}\mathbb{E}\left(Y-F_{Y}^{-1}\left(\alpha\right)\right)_{+}.
\]
Note now that, by Fubini's Theorem, 
\begin{align}
\mathcal{R}_{\sigma}\left(Y\right) & =\int_{0}^{1}\AVaR_{\alpha}\left(Y\right)\mu_{\sigma}\left(\mathrm{d}\alpha\right)\nonumber \\
 & =\int_{0}^{1}F_{Y}^{-1}\left(\alpha\right)+\frac{1}{1-\alpha}\mathbb{E}\left(Y-F_{Y}^{-1}\left(\alpha\right)\right)_{+}\mu_{\sigma}\left(\mathrm{d}\alpha\right)\nonumber \\
 & =\mathbb{E}\int_{0}^{1}\, F_{Y}^{-1}\left(\alpha\right)+\frac{1}{1-\alpha}\left(Y-F_{Y}^{-1}\left(\alpha\right)\right)_{+}\mu_{\sigma}\left(\mathrm{d}\alpha\right)\nonumber \\
 & =\mathbb{E}\, f_{0}\left(Y\right).\label{eq:f0}
\end{align}

To establish the assertion \eqref{eq:HS} it needs to be shown that
$\int_{0}^{1}f_{0}^{*}\left(\sigma\left(u\right)\right)\mathrm{d}u\le0$.
For this observe first that $f_{0}$ is almost everywhere differentiable
(because it is convex), with derivative  
\begin{eqnarray*}
f_{0}^{\prime}\left(y\right) & = & \int_{\left\{ \alpha\colon F_{Y}^{-1}\left(\alpha\right)\le y\right\} }\frac{1}{1-\alpha}\mu_{\sigma}\left(\mathrm{d}\alpha\right)\\
 & = & \int_{0}^{F_{Y}(y)}\frac{1}{1-\alpha}\mu_{\sigma}\left(\mathrm{d}\alpha\right)=\sigma\left(F_{Y}(y)\right)
\end{eqnarray*}
(almost everywhere) by relation \eqref{eq:3-1}. Moreover $f_{0}^{*}\left(\sigma(u)\right)=\sup_{y}\sigma(u)y-f_{0}\left(y\right)$,
the supremum being attained at every $y$ satisfying $\sigma(u)=f_{0}^{\prime}\left(y\right)=\sigma\left(F_{Y}(y)\right)$,
hence at $y=F_{Y}^{-1}\left(u\right)$, and it follows that  
\[
f_{0}^{*}\left(\sigma\left(u\right)\right)=\sigma\left(u\right)\cdot F_{Y}^{-1}\left(u\right)-f_{0}\left(F_{Y}^{-1}\left(u\right)\right).
\]
Now 
\begin{eqnarray*}
\int_{0}^{1}f_{0}^{*}\left(\sigma\left(u\right)\right)\mathrm{d}u & = & \int_{0}^{1}\sigma\left(u\right)\cdot F_{Y}^{-1}\left(u\right)\mathrm{d}u-\int_{0}^{1}f_{0}\left(F_{Y}^{-1}\left(u\right)\right)\mathrm{d}u\\
 & = & \mathcal{R}_{\sigma}\left(Y\right)-\mathbb{E}f_{0}\left(Y\right).
\end{eqnarray*}
But it was established already in \eqref{eq:f0} that $\mathcal{R}_{\sigma}\left(Y\right)=\mathbb{E}f_{0}\left(Y\right)$,
so that $\int_{0}^{1}f_{0}^{*}\left(\sigma\left(u\right)\right)\mathrm{d}u=0$.
This finally proves the second inequality. 
\end{proof}
The Average Value-at-Risk is a special case of the infimum in \eqref{eq:HS}.
Indeed, it follows from the proof that the infimum is attained at
a function of the form $f_{q}\left(y\right)=q+\frac{1}{1-\alpha}\left(y-q\right)_{+}$
with conjugate 
\[
f_{q}^{*}\left(x\right)=\begin{cases}
-q+q\, x & \text{ if }0\le x\le\frac{1}{1-\alpha}\\
\infty & \text{ else}.
\end{cases}
\]
It follows that $\int_{0}^{1}f_{0}^{*}\left(\sigma_{\alpha}\left(x\right)\right)\mathrm{d}x=\int_{0}^{\alpha}f_{0}^{*}\left(0\right)\mathrm{d}x+\int_{\alpha}^{1}f_{0}^{*}\left(\frac{1}{1-\alpha}\right)\mathrm{d}x=-\alpha q+\left(-q+\frac{q}{1-\alpha}\right)(1-\alpha)=0$,
such that 
\begin{equation}
\AVaR_{\alpha}\left(Y\right)=\inf_{q\in\mathbb{R}}\,\mathbb{E}f_{q}\left(Y\right)=\inf_{q}\: q+\frac{1}{1-\alpha}\mathbb{E}\left(Y-q\right)_{+}.\label{eq:AVaR}
\end{equation}
Clearly, the infimum in \eqref{eq:AVaR} is in $\mathbb{R}$, a much
smaller space than convex functions from $\mathbb{R}$ to $\mathbb{R}$,
as required in \eqref{eq:HS}.

\section{\label{sec:StochasticOptimization}Implications for Stochastic Optimization }

Portfolio optimization (i.e.\ asset allocation) is a typical problem
in stochastic optimization. In this section we outline how the inf-representation
\eqref{eq:HS} of spectral risk measures can be applied to efficiently
evaluate stochastic optimization problems involving spectral risk
measures.

\subsection*{Problem Formulation}

Stochastic optimization often considers the problem
\begin{equation}
\begin{array}{rl}
\text{minimize} & \mathbb{E}\, H\left(x,Y\right)\\
\text{subject to} & x\in\mathbb{X},
\end{array}\label{eq:SOE}
\end{equation}
where $x$ can be chosen in a decision space $\mathbb{X}$ and $H$
is a loss function; for $x$ given, $\omega\mapsto H\left(x,Y\left(\omega\right)\right)$
is a random variable for which the expectation in \eqref{eq:SOE}
can be maximized. (For the purpose of portfolio optimization, the
function $H$ is simply $H=x^{\top}Y$, where $Y$ collect the returns
of the individual stocks.)

If the function $H$ should be evaluated subject to risk as well,
then a frequently used substitute for \eqref{eq:SOE} is 
\begin{equation}
\begin{array}{rl}
\text{minimize} & \mathcal{R}\left(H\left(x,Y\right)\right)\\
\text{subject to} & x\in\mathbb{X}
\end{array},\label{eq:5}
\end{equation}
where $\mathcal{R}$ is a convex risk measure with the usual dual
representation 
\[
\mathcal{R}\left(Y\right)=\sup\left\{ \mathbb{E}YZ\colon\mathcal{R}^{*}\left(Z\right)<\infty\right\} .
\]
Combining these formulas leads to the combined problem 
\begin{equation}
\begin{array}{rl}
\text{minimize} & \sup\:\mathbb{E}\, H\left(x,Y\right)\cdot Z\\
\text{subject to} & x\in\mathbb{X}\\
 & \mathcal{R}^{*}\left(Z\right)<\infty.
\end{array}\label{eq:2}
\end{equation}
This problem formulation \eqref{eq:2} -- a mini-max problem -- is
not favorable for numerical optimizations, as for any function evaluation
for a given $x$ to compute the minimum a complete maximum has to
be evaluated first. Provided that there are efficient mathematical
estimates for the particular problem this may to be done with reduced
precision, as otherwise the results for the optimal $x\in\mathbb{X}$
is unpredictably wrong; if no such estimates are available, then the
supremum in \eqref{eq:2} has to be computed precisely -- impossible
for numerical compuatations.

For the Average Value-at-Risk, 
\[
\AVaR_{\alpha}\left(Y\right)=\frac{1}{1-\alpha}\int_{\alpha}^{1}F_{Y}^{-1}\left(p\right)\mathrm{d}p=\sup\left\{ \mathbb{E}YZ\colon0\le Z\le\frac{1}{1-\alpha}\right\} ,
\]
none of these formulations are useful in \eqref{eq:5} neither. However,
the representation 
\[
\AVaR_{\alpha}\left(Y\right)=\min_{q}q+\frac{1}{1-\alpha}\mathbb{E}\left(Y-q\right)_{+}
\]
can be employing immediately, as the problem \eqref{eq:5} rewrites
\begin{equation}
\begin{array}{rl}
\text{minimize} & q+\frac{1}{1-\alpha}\mathbb{E}\left(H(x,Y)-q\right)_{+}\\
\text{subject to} & x\in\mathbb{X}\\
 & q\in\mathbb{R},
\end{array}\label{eq:AVaR-1}
\end{equation}
which is just a minimization problem with the same objective and optimal
value as \eqref{eq:2}, but much easier to handle from implementation
and numerical point of view. 

For a spectral risk measure problem \eqref{eq:5} rewrites as 
\begin{equation}
\begin{array}{rl}
\text{minimize} & \mathbb{E}f\left(H\left(x,Y\right)\right)+\int_{0}^{1}f^{*}\left(\sigma\left(u\right)\right)\mathrm{d}u\\
\text{subject to} & x\in\mathbb{X}\\
 & f\text{ (convex)},
\end{array}\label{eq:convex}
\end{equation}
which is~-- in contrast to \eqref{eq:5}~-- a straight forward optimization
problem again, which allows an immediate implementation. The price,
which is to pay in contrast to the simple Average Value-at-Risk, is
that an entire function has to be looked up in \eqref{eq:convex},
whereas just an additional number ($q$) appears in \eqref{eq:AVaR-1}.

\subsection*{Implementation}

For numerical implementations to find approximations of \eqref{eq:convex}
it should be mentioned that $\sigma\in L^{\infty}$ can be approximated
by step functions $\overline{\sigma}_{n}$ such that $\sigma\le\bar{\sigma}_{n}$,
where $\bar{\sigma}_{n}=\sum_{i=1}^{n}\sigma_{i}\one_{\left[\alpha_{i-1},\alpha_{i}\right)}$
($0=\alpha_{0}<\dots\alpha_{i}<\alpha_{i+1}<\dots\alpha_{n}=1$, $0<\sigma_{i}<\sigma_{i+1}$).
In this situation the range of the approximation $\bar{\sigma}_{n}$
is the finite set $\left\{ \sigma_{1},\sigma_{2},\dots\sigma_{n}\right\} $
and 
\[
\mathcal{R}_{\sigma}\left(Y\right)=\int_{0}^{1}F_{Y}^{-1}\left(\alpha\right)\sigma\left(\alpha\right)\mathrm{d}\alpha\le\int_{0}^{1}F_{Y}^{-1}\left(\alpha\right)\bar{\sigma}_{n}\left(\alpha\right)\mathrm{d}\alpha.
\]
The corresponding measure on $\left[0,1\right]$ is $\mu_{\sigma_{n}}\left(A\right)=\sigma_{1}\delta_{0}+\sum_{i=1}^{n}\left(1-\alpha_{i}\right)\left(\sigma_{i+1}-\sigma_{i}\right)\delta_{\alpha_{i}}=:\sum_{i=0}^{n-1}\mu_{i}\delta_{\alpha_{i}}$.
From the proof of Theorem~\ref{thm:InfRep} it follows that one may
choose the ansatz 
\[
f_{n}\left(y\right):=\mu_{0}y+\sum_{i=1}^{n}\mu_{i}\left(q_{i}+\frac{1}{1-\alpha_{i}}\left(y-q_{i}\right)_{+}\right)
\]
 to solve the approximating problem 
\[
\begin{array}{rl}
\text{minimize} & \mathbb{E}f_{n}\left(H\left(x,Y\right)\right)\\
\text{subject to} & x\in\mathbb{X}\\
 & q_{1}\le q_{2}\le\dots q_{n}.
\end{array}
\]
The resulting value consequently is an upper bound of the problem
\eqref{eq:convex}.

\section{Summary}

This article outlines new descriptions of spectral risk measures.
Spectral risk measures constitute an integral subclass of risk measures,
as every risk measure can be written as a supremum of spectral risk
measures. 

The first representation derived is built as a supremum, based on
conjugate duality. The other representation, which is the central
result of this article, is described as an infimum. This description
makes spectral risk measures eligible for successful future use in
stochastic optimization.

\section{Acknowledgment}

We wish to thank the referees for their constructive criticism. 

\bibliographystyle{alpha}
\bibliography{../../Literatur/LiteraturAlois}

\appendix

\section*{Appendix}

For reference and the sake of completeness we list the following elementary
result for affine linear transformations of the convex conjugate function.
\begin{lem}
\label{lem:Transform}The convex conjugate of the function $g\left(x\right):=\alpha+\beta x+\gamma\cdot f\left(\lambda x+c\right)$
for $\gamma>0$ and $\lambda\neq0$ is 
\[
g^{*}\left(y\right)=-\alpha-c\,\frac{y-\beta}{\lambda}+\gamma\cdot f^{*}\left(\frac{y-\beta}{\lambda\gamma}\right).
\]
\end{lem}
\begin{proof}
Just observe that 
\begin{align}
g^{*}\left(y\right)= & \sup_{x}\, yx-g(x)\nonumber \\
= & \sup_{x}\, yx-\alpha-\beta x-\gamma\cdot f\left(\lambda x+c\right)\nonumber \\
= & \sup_{x}\, y\frac{x-c}{\lambda}-\alpha-\beta\frac{x-c}{\lambda}-\gamma\cdot f\left(x\right)\label{eq:drei}\\
= & -\alpha-c\frac{y-\beta}{\lambda}+\sup_{x}\, x\frac{y-\beta}{\lambda}-\gamma\cdot f\left(x\right)\nonumber \\
= & -\alpha-c\frac{y-\beta}{\lambda}+\gamma\cdot\sup_{x}\, x\frac{y-\beta}{\lambda\gamma}-f\left(x\right)\nonumber \\
= & -\alpha-c\frac{y-\beta}{\lambda}+\gamma\cdot f^{*}\left(\frac{y-\beta}{\lambda\gamma}\right),\nonumber 
\end{align}
where we have replaced $x$ by $\frac{x-c}{\lambda}$ in \eqref{eq:drei}.\end{proof}

\end{document}